\documentclass[11pt,a4paper]{amsart}
\usepackage{amsfonts,amssymb,amsbsy,amsmath,amsthm}
\usepackage[dvips,colorlinks=true,linkcolor=blue,urlcolor=red]{hyperref}
\newcommand{\bel}{\begin{equation} \label}

\newcommand{\ee}{\end{equation}}

\newcommand{\rd}{{\mathbb R}^{2}}
\newcommand{\re}{{\mathbb R}}

\newcommand{\tr}{{\rm Tr}\,}

\newcommand{\esssup}{\operatornamewithlimits{\text{ess-sup}}}
\newcommand{\car}{\mathbf{1}}
\newcommand{\R}{\mathbb{R}}

\newcommand{\N}{\mathbb{N}}

\newcommand{\Z}{\mathbb{Z}}

\newcommand{\vers}{\operatornamewithlimits{\to}}

\newcommand{\D}{\displaystyle}

\newcommand{\pro}{\mathbb{P}}

\theoremstyle{plain}
\newtheorem{Th}{Theorem}
\newtheorem{Le}{Lemma}

\theoremstyle{definition}

\title[Lifshitz tails in a constant magnetic field]{Lifshitz tails for
  alloy type models in a constant magnetic field} 
\author{Fr{\'e}d{\'e}ric Klopp} \address[Fr{\'e}d{\'e}ric Klopp]{LAGA, U.M.R. 7539
  C.N.R.S, Institut Galil{\'e}e, Universit{\'e} Paris-Nord, 99 Avenue
  J.-B. Cl{\'e}ment, F-93430 Villetaneuse, France\ et \ Institut
  Universitaire de France}
\email{\href{mailto:klopp@math.univ-paris13.fr}{klopp@math.univ-paris13.fr}}
\dedicatory{Dedicated to the memory of Pierre Duclos.}
\thanks{Part of this work was done during the conference ``Spectral
  analysis of differential operators'' held at the CIRM, Luminy
  (07/07-11/07/2008); it is a pleasure to thank P. Briet, F. Germinet
  and G. Raikov, the organizers of this event for their invitation to
  participate to the meeting.\\
  This work was partially supported by the grant ANR-08-BLAN-0261-01}
\begin{document}
% \french
%
\begin{abstract}
  In this paper, we study Lifshitz tails for a 2D Landau Hamiltonian
  perturbed by a random alloy-type potential constructed with single
  site potentials decaying at least at a Gaussian speed. We prove
  that, if the Landau level stays preserved as a band edge for the
  perturbed Hamiltonian, at the Landau levels, the integrated density
  of states has a Lifshitz behavior of the type $e^{-\log^2|E-2bq|}$.
  \vskip.5cm\noindent \textsc{R{\'e}sum{\'e}.} Dans cette note, nous
  d{\'e}montrons qu'en dimension 2, la densit{\'e} d'{\'e}tats int{\'e}gr{\'e}e d'un
  op{\'e}rateur de Landau avec un potentiel al{\'e}atoire non n{\'e}gatif de type
  Anderson dont le potentiel de simple site d{\'e}cro{\^\i}t au moins aussi
  vite qu'une fonction gaussienne admet en chaque niveau de Landau,
  disons, $2bq$, si celui-ci est un bord du spectre, une asymptotique
  de Lifshitz du type $e^{-\log^2|E-2bq|}$.
\end{abstract}
\setcounter{section}{-1}
\maketitle
\section{Introduction}
\label{sec:introduction}
On $C_0^{\infty}(\rd)$, consider the Landau Hamiltonian
\begin{equation*}
  H_0 = H_0(b): =  (-i\nabla -A)^2 - b
\end{equation*}
where $A = (-\frac{bx_2}{2},\frac{bx_1}{2})$ is the magnetic
potential, and $b>0$ is the constant scalar magnetic field. $H_0$ is
essentially self-adjoint on $C_0^{\infty}(\rd)$.  It is well-known
that $\sigma(H_0)$, the spectrum of the operator $H_0$, consists of
the so-called Landau levels $\{2bq;\ q\in\N=\{0,1,2,\cdots\}\}$; each
Landau level is an eigenvalue of infinite multiplicity of $H_0$.\\
Consider now the random ${\mathbb Z}^2$-ergodic alloy-type electric
potential
\begin{equation*}
  V_{\omega}(x): = \sum_{\gamma \in {\mathbb Z}^2}
  \omega_{\gamma} u(x - \gamma), \quad x \in \re^2
\end{equation*}
where we assume that
\begin{itemize}
\item ${\bf H}_1$: The single-site potential $u$ satisfies, for some
  $C>0$ and $x_0\in\rd$,

  \begin{equation*}
    \frac1{C}{\bf 1}_{\{x \in \rd\,;\, |x-x_0| < 1/C\}} \leq u(x) \leq
    Ce^{- |x|^2/C}.
  \end{equation*}
\item ${\bf H}_2$: The coupling constants
  $\left\{\omega_{\gamma}\right\}_{{\gamma} \in {\mathbb Z}^2}$ are
  non-trivial, almost surely bounded i.i.d. random variables.
\end{itemize}
These two assumptions guarantee $V_{\omega}$ is almost surely
bounded. On the domain of $H_0$, define the operator
$H_{\omega}:=H_0+V_{\omega}$. The integrated density of states (IDS)
of the operator $H_\omega$ is defined as the non-decreasing
left-continuous function ${\mathcal N} : \re \to [0,\infty)$ which,
almost surely, satisfies
\begin{equation*}
  \int_{\re} \varphi(E) d{\mathcal N}(E) = \lim_{R \to \infty} R^{-2}
  \,\tr\left({\bf 1}_{{\Lambda}_R} \varphi(H) {\bf 1}_{{\Lambda}_R} \right),
  \quad \forall \varphi \in C_0^{\infty}(\re).
\end{equation*}
Here and in the sequel, $\Lambda_R : = \left(-\frac{R}{2},
  \frac{R}{2}\right)^2$ and ${\bf 1}_{{\mathcal O}}$ denotes the
characteristic function of the set ${\mathcal O}$.\\
By the Pastur-Shubin formula (see e.g. \cite[Section 2]{MR2378428}),
we have
\begin{equation*}
  \int_{\re} \varphi(E) d{\mathcal N}(E) =
  {\mathbb E}\left(\tr \left({\bf 1}_{{\Lambda}_1} \varphi(H) {\bf
        1}_{{\Lambda}_1} \right) \right), \quad \forall \varphi \in
  C_0^{\infty}(\re), 
\end{equation*}
where ${\mathbb E}$ denotes the mathematical expectation with respect
to the random variables $(\omega_\gamma)_\gamma$. Moreover, there
exists a set $\Sigma\subset \re$ such that $\sigma(H_{\omega}) =
\Sigma$ almost surely. $\Sigma$ is the support of the positive measure
$d{\mathcal N}$. The aim of the present article is to study the
asymptotic behavior of ${\mathcal N}$ near the edges of $\Sigma$. It
is well known that, for many random models, this behavior is
characterized by a very fast decay which goes under the name of
``Lifshitz tails''. It was studied extensively (see
e.g.~\cite{MR94h:47068,MR2378428,MR2307751} and references therein).
\vskip.2cm\noindent In order to fix the picture of the almost sure
spectrum $\sigma(H_{\omega})$, we assume:
\begin{itemize}
\item ${\bf H}_3$: the common support of the random variables
  $(\omega_{\gamma})_{\gamma\in\Z^2}$ consists of the interval
  $[\omega_-,\omega_+]$ where $\omega_- <\omega_+$ and $\omega_-
  \omega_+=0$.
\item ${\bf H}_4$: $M_+ - M_- < 2b$ where
  \begin{equation*}
    \pm M_{\pm} : = \esssup_{\omega}\;\sup_{x \in \re^2} \; (\pm
    V_{\omega}(x)).
  \end{equation*}
\end{itemize}
Assumptions ${\bf H}_1$ -- ${\bf H}_4$ imply that $M_- M_+=0$. It also
implies that the union $\D\bigcup_{q=0}^{\infty} [2bq+M_-, 2bq+M_+]$,
which contains $\Sigma$, is disjoint. Let $W$ be the bounded
$\Z^2$-periodic potential defined by
\begin{equation*}
  W(x) : = \sum_{\gamma \in \Z^2} u(x-\gamma), \quad x \in\re^2.
\end{equation*}
On the domain of $H_0$, define the operators $H^{\pm} : = H_0 +
\omega_{\pm} W$. It is easy to see that
\begin{equation*}
  \sigma(H^-) \subseteq
  \cup_{q=0}^{\infty} [2bq+M_-, 2bq], \quad \sigma(H^+) \subseteq
  \cup_{q=0}^{\infty} [2bq, 2bq+M_+],  
\end{equation*}
and
\begin{equation*}
  \sigma(H^-) \cap
  [2bq+M_-, 2bq] \neq \emptyset, \quad \sigma(H^+) \cap [2bq, 2bq+M_+]
  \neq \emptyset, \quad \forall q\in\N.  
\end{equation*}
Set
\begin{equation*}
  E_q^- := \min(\partial\sigma(H^-) \cap [2bq+M_-, 2bq]),
  \quad
  E_q^+ := \max(\partial\sigma(H^+) \cap [2bq, 2bq+M_+]).
\end{equation*}
The standard characterization of the almost sure spectrum (see also
\cite[Theorem 5.35]{MR94h:47068}) yields
\begin{equation*}
  \Sigma =\bigcup_{q=0}^{\infty} [E_q^-,E_q^+],\quad E_q^-<E_q^+
\end{equation*}
i.e. $\Sigma$ is represented as a disjoint union of compact intervals,
and each interval $[E_q^-,E_q^+]$ contains exactly one Landau level
$2bq$. Actually, one has either $E_q^-=2bq$ or $E_q^+=2bq$; more
precisely $E_q^-=2bq$ if $\omega_-=0$ and $E_q^+=2bq$ if $\omega_+=0$.\\
In Theorem 2.1 of~\cite{MR2249786}, the authors describe the behavior
of ${\mathcal N}(2bq + E) - {\mathcal N}(2bq)$ when $E$ tends to $0$
while in $\Sigma$. Under the assumption that $u$ does not decay as
fast as in assumption ${\bf H}_1$, they compute the logarithmic
asymptotics of the IDS near $2bq$. Under assumption ${\bf H}_1$, the
authors obtained the optimal logarithmic upper bound and a lower bound
that they deemed not to be optimal. In our main result, we obtain the
optimal lower bound, thus, proving the logarithmic asymptotics.
\begin{Th}
  \label{thr:1}
  Let $b>0$ and assumptions ${\bf H}_1$ -- ${\bf H}_4$ hold. Assume
  that, for some $C > 0$ and $\kappa > 0$,
  \begin{equation}
    \label{fin70} 
    {\mathbb P}(|\omega_0|\leq E) \sim CE^{\kappa}, \quad E \downarrow
    0.
  \end{equation}
  Then, for any $q\in\N$, one has
  \begin{equation}
    \label{hin2} 
    \lim_{\substack{E\to0\\E\in\Sigma}}
    \frac{\ln{|\ln{({\mathcal N}(2bq + E) - {\mathcal
            N}_b(2bq})|}}{\ln{|\ln{E}|}}=2.
  \end{equation}
\end{Th}
\noindent Thus, Theorem~\ref{thr:1} states that, at the Landau level
$2bq$, when it is a spectral edge for $H_0$, the IDS decays roughly as
$e^{-\log^2|E-2bq|}$. This decay is faster than any power of
$|E-2bq|$. This explains why we name this behavior also Lifshitz tails
even though it is much slower than the Lifshitz tails obtained when
the magnetic field is absent (see e.g.~\cite{MR94h:47068}).\\
In~\cite{MR2249786}, the upper bound in~\eqref{hin2} is proved under
less restrictive assumptions; indeed, Theorem 5.1 of~\cite{MR2249786}
states in particular that, under our assumptions,
\begin{equation*}
  \limsup_{E \downarrow 0}
  \frac{\ln{|\ln{({\mathcal N}(2bq + E) - {\mathcal
          N}_b(2bq})|}}{\ln{|\ln{E}|}}\leq 2.
\end{equation*}
So it suffices to prove
\begin{equation}
  \label{eq:18}
  \liminf_{E \downarrow 0}
  \frac{\ln{|\ln{({\mathcal N}(2bq + E) - {\mathcal
          N}_b(2bq})|}}{\ln{|\ln{E}|}}\geq 2.
\end{equation}
The improvement over the results in~\cite{MR2249786} is obtained
through a different ana\-lysis that borrows ideas and estimates
from~\cite{MR2097581}. The basic idea is to show that, for energies at
a distance at most $E$ from $2bq$, the single site potential $u$ can
be replaced by an effective potential that has a support of size
approximately $|\log E|^{1/2}$ (see section~\ref{sec:proof-theorem}
and Lemma~\ref{le2} therein). This can then be used to estimate the
probability of the occurrence of such energies.
\section{Periodic approximation}
\label{sec:peri-appr}
\noindent Assume that hypotheses ${\bf H}_1-{\bf H}_4$ hold. For the
sake of definiteness, from now on, we assume that $\omega_-=0$. So,
for $q\in\N$, we have $E_q^-=2bq$. Moreover,~\eqref{fin70} becomes
${\mathbb P}(0\leq\omega_0\leq E)\sim CE^{\kappa}$ for $E>0$ small and
${\mathbb P}(0\geq\omega_0\geq -E)=0$ for any $E>0$. Up to obvious
modifications, the case $\omega_+=0$ is dealt with in the same way.\\
We now recall some useful results from~\cite{MR2249786}. Pick $a>0$
such that $\frac{ba^2}{2\pi} \in {\mathbb N}$. Set $L : = (2n+1)a/2$,
$n\in{\mathbb N}$, and define the random $2L{\mathbb Z}^2$-periodic
potential
\begin{equation}
  \label{eq:20}
  V_{L,\omega}^{\rm per}(x) =
  V_{L,\omega}^{\rm per}(x): = \sum_{\gamma \in 2L{\mathbb
      Z}^2}\left(V_{\omega} {\bf
      1}_{\Lambda_{2L}}\right)(x+\gamma), \quad x \in \re^2.  
\end{equation}
For $q\in\N$, let $\Pi_q$ be the orthogonal projection onto the
$(q+1)$-st Landau level i.e. the orthogonal projection onto ${\rm
  Ker}(H_0 - 2bq)$. Consider the bounded operator
$\Pi_qV_{L,\omega}^{\text{per}}\Pi_q$. It is invariant by the Abelian
group of magnetic translations generated by $2L{\mathbb Z}^2$ (see
section 2 in~\cite{MR2249786}). Hence,
$\Pi_qV_{L,\omega}^{\text{per}}\Pi_q$ admits an integrated density of
states that we denote by $\rho_{q,L,\omega}(E)$
(see~\cite{MR2249786}). In~\cite{MR2249786}, we have proved
\begin{Th}[\cite{MR2249786}]
  \label{f31}
  Assume that hypotheses ${\bf H}_1-{\bf H}_4$ hold and $\omega_-=0$.
  Pick $q\in\N$ and $\eta>0$. Then, there exist $\nu> 0$, $C>1$ and
  $E_0>0$, such that for each $E \in(0,E_0)$ and $L \geq E^{-\nu}$, we
  have
  \begin{multline*}
    {\mathbb E}\left(\rho_{q,L,\omega}(E/C)\right) - e^{-E^{-\eta}}
    \leq {\mathcal N}(2bq + E) - {\mathcal N}(2bq)\\ \leq {\mathbb
      E}\left(\rho_{q,L,\omega}(C E)\right) + e^{-E^{-\eta}}.
  \end{multline*}
\end{Th}
\noindent As $\rho_{q,L,\omega}(E)$ is the IDS of the periodic
operator $\Pi_q V_{L,\omega}^{\text{per}}\Pi_q$ at energy $E$, it
vanishes if and only if $\sigma(\Pi_q V_{L,\omega}^{\rm per}
\Pi_q)\cap(-\infty,E]\not=\emptyset$. Moreover, $\rho_{q,L,\omega}(E)$
is bounded by $CL^d$ where the constant $C$ is locally uniform in $E$
(see~\cite{MR2249786}). Thus, we get that, for some $C>0$,
\begin{equation}
  \label{eq:10}
  {\mathbb E}\left(\rho_{q,L,\omega}(E)\right)\leq C
  \,L^d\,\pro\left(\sigma(\Pi_q V_{L,\omega}^{\rm per}  
    \Pi_q)\cap(-\infty,CE]\not=\emptyset\right).
\end{equation}
Then, the estimate~\eqref{eq:18} and, thus, Theorem~\ref{thr:1}, is a
consequence of
\begin{Th}
  \label{thr:2}
  For $\eta\in(0,1)$, there exists $C_\eta>0$ such that, for $E$
  sufficiently small and $L\geq1$, one has, for almost all $\omega$,
  \begin{multline}
    \label{eq:19}
    e^{|\log E|^{1-\eta}\log|\log E\|}\,\Pi_q V_{L,\omega}^{\rm per}
    \Pi_q\\\geq
    \left[\inf_{\gamma\in\Lambda_{2L}\cap\Z^2}\left(\sum_{|\beta-\gamma|\leq
          |\log E|^{(1-\eta)/2}}\omega_\beta \right) - e^{-|\log
        E|^{1-\eta}/C_\eta}\right]\Pi_q.
  \end{multline}
\end{Th}
\noindent The proof of Theorem~\ref{thr:2} relies on Lemma~\ref{le2}
which shows that, at the expense of a small error in energy, we can
``enlarge'' the support of the single site potential
$u$. Lemma~\ref{le2} is stated and proved in
section~\ref{sec:proof-theorem}.\\
Let us now use Theorem~\ref{thr:2} to complete the proof
of~\eqref{eq:18} and, thus, of Theorem~\ref{thr:1}. Pick $L\asymp E
^{-\nu}$, $\nu$ given by Theorem~\ref{f31} and fix $\eta\in(0,1)$
arbitrary. Thus, by Theorem~\ref{f31},~\eqref{eq:10} implies that, for
$E>0$ small,
\begin{multline}
  \label{eq:12}
  {\mathcal N}(2bq + E)- {\mathcal N}(2bq) \\\leq
  CL^d\pro\left(\sigma(\Pi_q V_{L,\omega}^{\rm per}
    \Pi_q)\cap(-\infty,CE]\not=\emptyset\right) + e^{-E^{-\eta}}.
\end{multline}
Using~\eqref{eq:19}, as the random variables
$(\omega_\gamma)_{\gamma\in\Z^2}$ are i.i.d., for $E>0$ small, we
compute
\begin{equation}
  \label{eq:13}
  \begin{split}
    &\pro\left(\sigma(\Pi_q V_{L,\omega}^{\rm per}
      \Pi_q)\cap(-\infty,CE]\not=\emptyset\right)\\
    &\leq \pro\left(\inf_{\gamma\in\Lambda_{2L}\cap\Z^2}\left[
        \sum_{|\beta-\gamma|\leq |\log E|^{(1-\eta)/2}}\omega_\beta
      \right]-e^{-|\log E|^{1-\eta}/C_\eta}
      \leq e^{-|\log E|/2}\right)\\
    &\leq C\,L^d\,\pro\left(\sum_{|\beta|\leq |\log
        E|^{(1-\eta)/2}}\omega_\beta \leq 2e^{-|\log
        E|^{1-\eta}/C_\eta}\right).
  \end{split}
\end{equation}
Recall that, by~\eqref{fin70}, as $\omega_-=0$, one has ${\mathbb
  P}(0\leq\omega_0\leq E)\sim CE^{\kappa}$ and ${\mathbb
  P}(0\geq\omega_0\geq -E)=0$ for $E>0$ small. Hence, by a classical
standard large deviation result (see e.g.~\cite{MR1619036}), we obtain
that
\begin{equation*}
  \pro\left(\sum_{|\beta|\leq |\log
      E|^{(1-\eta)/2}}\omega_\beta \leq 2e^{-|\log
      E|^{1-\eta}/C_\eta}\right)\leq C_\eta e^{-|\log
    E|^{2-2\eta}/C_\eta}.
\end{equation*}
Thus, as $L\asymp E ^{-\nu}$, this,~\eqref{eq:12} and~\eqref{eq:13}
yield, for $E>0$ small,
\begin{equation*}
  {\mathcal N}(2bq + E)-{\mathcal N}(2bq) \leq C_\eta e^{-|\log
    E|^{2-2\eta}/C_\eta}.
\end{equation*}
As this bound holds for any $\eta>0$, we obtain~\eqref{eq:18} and,
thus, complete the proof of Theorem~\ref{thr:1}.\qed
\section{The proof of Theorem~\ref{thr:2}}
\label{sec:proof-theorem}
Recall that, for $q\in\N$, $\Pi_q$ is the orthogonal projection on the
eigenspace of $H_0$ corresponding to $2bq$, the $(q+1)$-st Landau
level of $H_0$. We recall
\begin{Le}[\cite{MR1939760}]
  \label{le1}
  Pick $p>1$ and let $V \in L^p({\mathbb R}^2)$ be radially
  symmetric.\\
  Let $(\mu_{q,k}(V))_{k \in \N}$ be the eigenvalues of the compact
  operator $\Pi_q V \Pi_q$ repeated according to
  multiplicity.\\
  Then, for $k\in \N$, one has
  \begin{equation*}
    \mu_{q,k}(V)=\langle V \varphi_{q,k},\varphi_{q,k} \rangle
  \end{equation*}
  where
  \begin{itemize}
  \item the functions $\varphi_{q,k}$ are given by
    \begin{equation*}
      \varphi_{q,k}(x) : = \sqrt{\frac{q!}{\pi\,k!}}
      \left(\frac{b}{2}\right)^{(k-q+1)/2} (x_1 + ix_2)^{k-q} {\rm
        L}^{(k-q)}_q \left(b|x|^2/2\right)e^{-b|x|^2/4},  
    \end{equation*}
    for $x=(x_1,x_2)\in {\mathbb R}^2$,
  \item ${\rm L}^{(k-q)}_q$ are the generalized Laguerre polynomials
    given by
    \begin{equation*}
      {\rm L}^{(k-q)}_q(\xi) : = \sum_{l={\rm max}\{0,q-k\}}^q
      \binom{k}{q-l} \frac{(-\xi)^l}{l!}, \quad \xi \geq 0, \quad q \in
      \N, \quad k \in \N,
    \end{equation*}
  \item $\langle\cdot,\cdot\rangle$ denotes the scalar product in
    $L^2({\mathbb R}^2)$.
  \end{itemize}
  Finally, for $k \in \N$, a normalized eigenfunctions of $\Pi_q V
  \Pi_q$ corresponding to the eigenvalue $\mu_{q,k}(V)$ is equal to
  $\varphi_{q,k}$. In particular, the eigenfunctions are independent
  of $V$.
\end{Le}
\noindent We denote by $D(x,R)$ the disk of radius $R>0$, centered at
$x \in {\mathbb R}^2$.  We set $\nu_{q,k}(R):
=\mu_{q,k}(\car_{D(0,R)})$ where $\car_A$ is the characteristic
function of the set $A$.
\begin{Le}
  \label{f1}
  Fix $q\in\N$. Define $\varrho=\varrho(R):=bR^2/2$ and
  \begin{equation}
    \label{eq:1}
    \nu^0_{q,k}(R)=\frac{e^{-\varrho} \varrho^{-q+1+k}}{q!}
    \frac{(k-\rho)^{2q-1}}{k!}.
  \end{equation}
  Pick $\beta\in(0,2)$. Let $f:[1,+\infty)\to[1,+\infty)$ be such that
  \begin{equation}
    \label{eq:8}
    % \left(k^{1-\frac1{2q}}+k^{\frac1{\beta}}\right) 
    % f^{-1}(k)\vers_{k\to+\infty}0.
    k^{2q-1} f^{-2q}(k)+k\,f^{-\beta}(k)
    \vers_{k\to+\infty}0.
  \end{equation}
  Then, there exists $k_0\geq1$ and $C>0$ such that, for $k\geq k_0$,
  \begin{equation}
    \label{2}
    \sup_{\substack{R>0\\ \varrho(R)\leq k-f(k)}}
    \left|\frac{\nu_{q,k}(R)}{\nu^0_{q,k}(R)}-
      1\right|\leq C\left(\frac{k^{2q-1}}{f^{2q}(k)}
      +\frac{k}{f^{\beta+1}(k)} \right).
  \end{equation}
\end{Le}
\noindent This lemma is an extension of Corollary 2
in~\cite{MR2097581} to a larger range of radii $R$.
\begin{proof}[Proof of Lemma~\ref{f1}]
  By Lemma~\ref{le1}, passing to polar coordinates $(r,\theta)$ in the
  integral $\langle \car_{D(0,R)} \varphi_{q,k}, \varphi_{q,k}
  \rangle$ and changing the variable $br^2/2 = \xi$, the eigenvalues
  $\nu_{q,k}(R)$ of the operator $\Pi_q \car_{D(0,R)} \Pi_q$ are
  written as
  \begin{equation*}
    % \label{1}
    \nu_{q,k}(R)  =
    \frac{q!}{k!} \int_0^{\varrho} \xi^k \,
    \left[{\rm L}^{(k-q)}_q(\xi)\right]^2 \, e^{-\xi} \, d\xi.
  \end{equation*}
  For $q=0$, we have
  \begin{equation}
    \label{eq:9}
    \nu_{0,k}(R)=
    \frac{1}{k!} \int_0^{\varrho} \xi^k \,
    e^{-\xi} \, d\xi=
    \frac{e^{-\varrho} \varrho^{k+1}}{k!} \int_0^1 e^{\rho t+k\log(1-t)}\,
    dt.
  \end{equation}
  Now, using a Taylor expansion at $0$ and the concavity of $t\mapsto
  \rho t+k\log(1-t)$, write
  \begin{equation*}
    \begin{split}
      \int_0^1 e^{\rho t+k\log(1-t)}\,
      dt&=\int_0^{(k-\rho)^{-\beta/2}}e^{\rho t+k\log(1-t)}dt+
      \int_{(k-\rho)^{-\beta/2}}^1e^{\rho t+k\log(1-t)}dt\\
      &=\int_0^{(k-\rho)^{-\beta/2}}e^{-(k-\rho)t}
      \left(1+O(k(k-\rho)^{-\beta})\right)dt\\& \hskip6cm+
      O\left(e^{-(k-\rho)^{1-\beta/2}}\right)\\&=\frac1{k-\rho}+
      O\left(\frac{k}{(k-\rho)^{\beta+1}}\right)
    \end{split}
  \end{equation*}
  This and~\eqref{eq:9} yields~\eqref{2} when $q=0$.\\
  Consider now the case $q\geq1$. For some $C_q>0$, one has
  \begin{equation}
    \label{eq:4}
    \forall k\geq1,\quad \sup_{s\in\{0,\ldots,q\}}\left|k^{s-q}
      \binom{k}{q-s}(q-s)!-1\right|\leq\frac{C_q}{k}.
  \end{equation}
  In order to check~\eqref{2}, we assume that $k \geq q$. In this
  case, using~\eqref{eq:4}, we compute
  \begin{equation}
    \label{eq:5}
    \begin{split}
      \nu_{q,k}(R) &= \frac{q!}{k!} \sum_{l,m=0}^q (-1)^{l+m}
      \frac{1}{m!l!}\binom{k}{q-l} \binom{k}{q-m} \int_0^{\varrho}
      e^{-\xi} \xi^{k-q+m+l} d\xi\\ &=V(k,q)+R(k,q)
    \end{split}
  \end{equation}
  where
  \begin{gather}
    \label{eq:2}
    \begin{split}
      V(k,q)&=\frac{1}{k!q!} \sum_{l,m=0}^q (-1)^{l+m} \binom{q}{l}
      \binom{q}{m} k^{2q-l-m}\int_0^{\varrho} e^{-\xi} \xi^{k-q+m+l}
      d\xi\\&=\frac{1}{k!q!}\int_0^{\varrho}
      e^{-\xi}\xi^{k-q}\left(k-\xi\right)^{2q} d\xi,
    \end{split}\\\intertext{and}
    \label{eq:3}
    \begin{split}
      |R(k,q)|&\leq \frac{C_q}{k}\frac{1}{k!q!} \sum_{l,m=0}^q
      \binom{q}{l} \binom{q}{m} k^{2q-l-m}\int_0^{\varrho} e^{-\xi}
      \xi^{k-q+m+l} d\xi\\&\leq \frac{C_q}{k}\frac{1}{k!q!}
      \int_0^{\varrho} e^{-\xi} \xi^{k-q}\left(k+\xi\right)^{2q} d\xi.
    \end{split}
  \end{gather}
  For $\rho\leq k-f(k)$, using~\eqref{eq:8}, one computes
  \begin{equation}
    \label{eq:6}
    \left|\frac{|R(k,q)|}{V(k,q)}\right|\leq
    C\frac{k^{2q-1}}{f^{2q}(k)}\vers_{k\to+\infty}0.
  \end{equation}
  On the other hand, as in the case $q=0$, we have
  \begin{equation*}
    \int_0^{\varrho}
    e^{-\xi}\xi^{k-q}\left(k-\xi\right)^{2q} d\xi=
    e^{-\rho}\rho^{k-q+1}(k-\rho)^{2q}I(k,\rho)
  \end{equation*}
  where
  \begin{equation*}
    I(k,\rho)=\int_0^1
    e^{\rho\xi}\,(1-\xi)^{k-q}\left(1+\frac{\rho}{k-\rho}\xi\right)^{2q}
    d\xi.
  \end{equation*}
  The function $\D t\mapsto \rho
  t+(k-q)\log(1-t)+2q\log\left(1+\frac{\rho}{k-\rho}t\right)$ is
  concave on $[0,1]$ and its derivative at $0$ is
  \begin{equation*}
    \rho-k+q+2q\rho/(k-\rho)=(\rho-k)\left(1+O(k(k-\rho)^{-2})\right).
  \end{equation*}
  Hence, as in the case $q=0$, we obtain that
  \begin{equation*}
    I(k,\rho)=\frac1{k-\rho}+ O\left(\frac{k}{(k-\rho)^{\beta+1}}\right).
  \end{equation*}
  Plugging this into~\eqref{eq:2}, using~\eqref{eq:6}
  and~\eqref{eq:3}, and replacing in~\eqref{eq:5}, we obtain~\eqref{2}
  for $q\geq1$.\\
  This completes the proof of Lemma~\ref{f1}.
\end{proof}
\noindent We will now use Lemma~\ref{f1} to derive the ``enlargement
of obstacles'' lemma for the Landau-Anderson model; we prove
\begin{Le}
  \label{le2}
  Let $q\in\N$ and fix $b>0$. Fix $\varepsilon>0$. There exists
  $C_0>0$ and $R_0>1$ such that, for each $R\geq R_0$,
  \begin{equation}
    \label{eq:7}
    \Pi_q \car_{D(0,\varepsilon)} \Pi_q \geq
    e^{-C_0 R^2\log R} \left(\Pi_q \car_{D(0,R)}\Pi_q - e^{-R^2/C_0} \Pi_q
      \car_{D(0,2R)} \Pi_q\right).
  \end{equation}
\end{Le}
\noindent This lemma is basically Lemma 2 in~\cite{MR2097581} except
that we want to control the behavior of the constants coming up in the
inequality in terms of $R$.
\begin{proof}[Proof of Lemma~\ref{le2}]
  We fix $\delta \in (0,1)$. Recall Lemma~\ref{f1}, in
  particular~\eqref{2} and~\eqref{eq:1}. Pick $C>2b$ and set
  $k_0=k_0(R):=C R^2$. Let $f$ satisfy~\eqref{eq:8}. Hence, there
  exists $R_0>0$ such that, for $R\geq R_0$ and $k\geq k_0=k_0(R)$,
  one has $k-f(k)\geq \rho(R)$. Thus, Lemma~\ref{f1} implies that, for
  $\tilde R\in[R/2,2R]$, one has
  \begin{equation}
    \label{eq:11}
    \begin{split}
      (1-\delta) \frac{e^{-\varrho(\tilde R)} \varrho(\tilde
        R)^{k-q+1}}{q!}&
      \frac{(k-\rho(\tilde R))^{2q-1}}{k!} \leq  \nu_{q,k}(\tilde R)\\
      &\leq (1+\delta) \frac{e^{-\varrho(\tilde R)} \varrho(\tilde
        R)^{k-q+1}}{q!}  \frac{(k-\rho(\tilde R))^{2q-1}}{k!}.
    \end{split}
  \end{equation}
  We show that, if $R\geq R_0$, then, the operator inequality
  \begin{equation}
    \label{7}
    \Pi_q \car_{D(0,\varepsilon)} \Pi_q \geq
    C_1 \left(\Pi_q \car_{D(0,R)} \Pi_q - C_2  \Pi_q
      \car_{D(0,2 R)} \Pi_q\right)
  \end{equation}
  holds with the following constants:
  \begin{itemize}
  \item
    \begin{equation}
      \label{6c}
      C_1:= \min_{k \in \{0,\ldots,k_0\}}
      \frac{\nu_{q,k}(\varepsilon)}{\nu_{q,k}(R)}\geq \frac1{C_0}
      e^{-2CR^2\log R}\,;
    \end{equation}
    the lower bound holds for sufficiently large $R$ and, as $k_0=C
    R^2$, is a consequence of~\eqref{eq:1} and~\eqref{2} written for
    $\nu_{q,k}(\varepsilon)$;
  \item
    \begin{equation}
      \label{6d}
      C_{2,q} : = \frac{1+\delta}{1-\delta}
      \left(\frac{C}{C-2b}\right)^{2q-1}
      2^{-2(k_0 - q + 1)} e^{-\varrho(R) + \varrho(2 R)}\leq
      e^{-R^2/C_0};
    \end{equation}
    the upper bound holds for sufficiently large $R$ and follows from
    $k_0\geq \rho(2R)$.
  \end{itemize}
  By Lemma \ref{le1}, the operators
  $\Pi_q\car_{D(0,\varepsilon)}\Pi_q$, $\Pi_q\car_{D(0,R)}\Pi_q$, and
  $\Pi_q\car_{D(0,2R)}\Pi_q$, are reducible in the same basis
  $\{\varphi_{q,k}\}_{k\in\N}$.  Hence, in order to prove \eqref{7},
  it suffices to check that, for each $k\in\N$, the following
  numerical inequality holds
  \begin{equation}
    \label{8}
    \nu_{q,k}(\varepsilon)
    \geq C_1 \left(\nu_{q,k}(R) - C_{2,q} \;\nu_{q,k}(2 R) \right).
  \end{equation}
  If $k \leq k_0$, then~\eqref{8} holds as $\nu_{q,k}(\varepsilon)
  \geq C_1 \nu_{q,k}(R)$ by \eqref{6c}. As $k_0\leq C R^2$ for $C>2b$
  and $\rho=bR^2/2$, or $k\geq k_0$, one has $C(k-\rho(2R))\geq
  (C-2b)(k-\rho(R))$.\\
  Thus, by~\eqref{eq:11} and~\eqref{6d}, we have
  \begin{equation*}
    \begin{split}
      \nu_{q,k}(R)& - C_{2,q}\nu_{q,k}(2 R) \leq (1+\delta)
      \frac{e^{-\varrho(R)} \varrho(R)^{-q+1}}{q!}
      \frac{(k-\rho(R))^{2q-1}\varrho(R)^k}{k!}  \\
      & - \left( \frac{1+\delta}{1-\delta} 2^{-2(k_0 - q+1)}
        e^{-\varrho(R) + \varrho(2
          R)}\left(\frac{C(k-\rho(2R))}{(C-2b)(k-\rho(R))}\right)^{2q-1}
      \right.\\
      &\hskip2cm\left. \times (1-\delta) \frac{e^{-\varrho(2 R)} \varrho(2
          R)^{-q+1}}{q!} \frac{(k-\rho(R))^{2q-1}\varrho(2 R)^k}{k!}
      \right) \\
      & = (1+\delta) \frac{e^{-\varrho(R)}} {q!}
      \frac{(k-\rho(R))^{2q-1} \varrho(2 R)^{k-q+1}}{k!}  2^{2(q-1)}
      \left( 2^{-2k} - 2^{-2k_0}\right).
    \end{split}
  \end{equation*}
  Hence, we find that $\nu_{q,k}(R) - C_2 \nu_{q,k}(2 R) \leq 0$ if
  $k\geq k_0$, which again implies \eqref{8}.  This completes the
  proof of Lemma~\ref{le2}.
\end{proof}
\noindent We now prove Theorem~\ref{thr:2}.\\
The magnetic translations for the constant magnetic field problem in
two-dimensions are defined as follows (see e.g.~\cite{Sj:91}). For any
field strength $b \in \R$, any vector $\alpha \in \R^2$, the magnetic
translation by $\alpha$, say, $U_\alpha^b$ is defined as
\begin{equation*}
  U_\alpha^b f(x):=e^{\frac{ib}{2} ( x_1 \alpha_2 - x_2 \alpha_1 )}
  f(x + \alpha)\quad f\in C_0^\infty ( \R^2).
\end{equation*}
The invariance of $H_0$ with respect to the group of magnetic
translations $(U_\alpha^b)_{\alpha\in{\mathbb Z}^2}$ implies that,
for $\gamma\in \Z^2$, one has
\begin{equation}
  \label{eq:14}
  U_\gamma^b \Pi_q \car_{D(0,\varepsilon)} \Pi_q U_{-\gamma}^b =
  \Pi_q \car_{D(\gamma,\varepsilon)} \Pi_q .
\end{equation}
Hypothesis ${\bf H}_1$ on the single-site potential $u$ guarantees
that there exists $\epsilon\in(0,1/2)$ so that $\D V_\omega\geq
\sum_{\gamma\in\Z^2}\omega_\gamma\car_{D(\gamma,\varepsilon)}$. Plugging
this into~\eqref{eq:20}, we get
\begin{equation}
  \label{eq:15}
  V_{L,\omega}^{\rm per}\geq \sum_{\gamma \in 2L{\mathbb
      Z}^2}\sum_{\beta\in\Lambda_{2L}\cap\Z^2}\omega_\beta
  \car_{D(\gamma+\beta,\varepsilon)}.
\end{equation}
Fix $\eta\in(0,1)$ and pick $R\asymp|\log E|^{(1-\eta)/2}$.
Lemma~\ref{le2} and~\eqref{eq:14} imply that
\begin{equation*}
  \begin{split}
    \Pi_q \car_{D(\gamma,\varepsilon)} \Pi_q \geq e^{-C_0 R^2\log R}
    \left(\Pi_q \car_{D(\gamma,R)}\Pi_q - e^{-R^2/C_0} \Pi_q
      \car_{D(\gamma,2R)} \Pi_q\right).
  \end{split}
\end{equation*}
Hence, as the random variables $(\omega_\gamma)_{\gamma\in\Z^2}$ are
bounded, this and~\eqref{eq:15} imply that
\begin{equation*}
  \begin{split}
    e^{C_0 R^2\log R}&\Pi_q V_{L,\omega}^{\rm per}\Pi_q \geq e^{C_0
      R^2\log R}\sum_{\gamma \in 2L{\mathbb
        Z}^2}\sum_{\beta\in\Lambda_{2L}\cap\Z^2}\omega_\beta\Pi_q
    \car_{D(\gamma+\beta,\varepsilon)}\Pi_q\\
    &\geq \sum_{\gamma \in 2L{\mathbb
        Z}^2}\sum_{\beta\in\Lambda_{2L}\cap\Z^2} \omega_\beta\Pi_q
    \car_{D(\gamma+\beta,R)}\Pi_q\\&\hskip2cm - Ce^{-R^2/C_0}
    \sum_{\gamma \in 2L{\mathbb
        Z}^2}\sum_{\beta\in\Lambda_{2L}\cap\Z^2}
    \Pi_q \car_{D(\gamma+\beta,2R)} \Pi_q\\
    &\geq \Pi_q\sum_{\gamma \in 2L{\mathbb
        Z}^2}\sum_{\beta\in\Lambda_{2L}\cap\Z^2} \omega_\beta
    \sum_{|\nu-\gamma-\beta|\leq R/2}\car_{|x-\nu|\leq1/2}\Pi_q-
    Ce^{-R^2/C_0} R^2\Pi_q\\
    &\geq    \left(\inf_{\gamma\in\Lambda_{2L}\cap\Z^2}
      \left(\sum_{|\beta-\gamma|\leq
          R/2}\omega_\beta \right)- CR^2e^{-R^2/C_0}\right)\Pi_q
  \end{split}
\end{equation*}
Taking into account $R\asymp|\log E|^{(1-\eta)/2}$, this completes the
proof of Theorem~\ref{thr:2}.
\def\cprime{$'$} \def\cydot{\leavevmode\raise.4ex\hbox{.}}
\def\cprime{$'$}


\begin{thebibliography}{1}

\bibitem{MR2097581} Jean-Michel Combes, Peter~D. Hislop, Fr{\'e}d{\'e}ric
  Klopp, and Georgi Raikov.  \newblock Global continuity of the
  integrated density of states for random {L}andau {H}amiltonians.
  \newblock {\em Comm. Partial Differential Equations},
  29(7-8):1187--1213, 2004.

\bibitem{MR1619036} Amir Dembo and Ofer Zeitouni.  \newblock {\em
    Large deviations techniques and applications}, volume~38 of {\em
    Applications of Mathematics (New York)}.  \newblock
  Springer-Verlag, New York, second edition, 1998.

\bibitem{MR82c:81181} B.~A. Dubrovin and S.~P. Novikov.  \newblock
  Fundamental states in a periodic field. {M}agnetic {B}loch functions
  and vector bundles.  \newblock {\em Dokl. Akad. Nauk SSSR},
  253(6):1293--1297, 1980.

\bibitem{MR2307751} Werner Kirsch and Bernd Metzger.  \newblock The
  integrated density of states for random {S}chr{\"o}dinger operators.
  \newblock In {\em Spectral theory and mathematical physics: a
    {F}estschrift in honor of {B}arry {S}imon's 60th birthday},
  volume~76 of {\em Proc. Sympos.  Pure Math.}, pages
  649--696. Amer. Math. Soc., Providence, RI, 2007.

\bibitem{MR2249786} Fr{\'e}d{\'e}ric Klopp and Georgi Raikov.  \newblock
  Lifshitz tails in constant magnetic fields.  \newblock {\em
    Comm. Math. Phys.}, 267(3):669--701, 2006.

\bibitem{MR94h:47068} Leonid Pastur and Alexander Figotin.  \newblock
  {\em Spectra of random and almost-periodic operators}, volume 297 of
  {\em Grundlehren der Mathematischen Wissenschaften [Fundamental
    Principles of Mathematical Sciences]}.  \newblock Springer-Verlag,
  Berlin, 1992.

\bibitem{MR1939760} Georgi~D. Raikov and Simone Warzel.  \newblock
  Quasi-classical versus non-classical spectral asymptotics for
  magnetic {S}chr{\"o}dinger operators with decreasing electric
  potentials.  \newblock {\em Rev. Math. Phys.}, 14(10):1051--1072,
  2002.

\bibitem{Sj:91} Johannes~Sj{\"o}strand.  \newblock Microlocal analysis for
  periodic magnetic {Schr{\"o}dinger} equation and related questions.
  \newblock In {\em Microlocal analysis and applications}, volume 1495
  of {\em Lecture Notes in Mathematics}, Berlin, 1991. Springer
  Verlag.

\bibitem{MR2378428} Ivan Veseli{\'c}.  \newblock {\em Existence and
    regularity properties of the integrated density of states of
    random {S}chr{\"o}dinger operators}, volume 1917 of {\em Lecture Notes
    in Mathematics}.  \newblock Springer-Verlag, Berlin, 2008.

\end{thebibliography}
\end{document}